\documentclass{article}
\usepackage[utf8]{inputenc}
\usepackage[margin=1.25in]{geometry}
\usepackage{amsmath, amssymb, amsthm}
\usepackage{tikz-cd}
\usepackage[all]{xy}

\newcommand{\id}{\mathrm{id}}

\usepackage{graphicx}
\usepackage[all]{xy}
\usepackage{enumerate}

\numberwithin{equation}{section}

\theoremstyle{plain}
\newtheorem{theorem}{Theorem}[section]
\newtheorem{lemma}[theorem]{Lemma}
\newtheorem{proposition}[theorem]{Proposition}
\newtheorem{corollary}[theorem]{Corollary}

\theoremstyle{definition}
\newtheorem{definition}[theorem]{Definition}
\newtheorem{notation}[theorem]{Notation}

\newtheorem{remark}[theorem]{Remark}

\newtheorem{example}[theorem]{Example}

\newtheorem{openproblem}[theorem]{Open Problem}
\newtheorem{convention}[theorem]{Convention}

\usepackage{tikz}
\usepackage{quiver}

\renewcommand{\epsilon}{\varepsilon}
\newcommand{\el}{\mathrm{el}}
\newcommand{\down}{\downarrow \hspace*{-0.6mm} }

\newcommand{\catc}{category, }

\newcommand{\prs}{presheaf }

\newcommand{\fg}{finitely generated }

\newcommand{\fgp}{finitely generated. }
\newcommand{\fp}{finitely presentable }
\newcommand{\fpc}{finitely presentable, }
\newcommand{\fpp}{finitely presentable. }
\newcommand{\gr}{graduated }

\newcommand{\cata}{\mathcal{A}}

\newcommand{\op}{\mathrm{op}}
\newcommand{\Set}{\mathbf{Set}}

\usepackage[dvipsnames,svgnames]{xcolor}
\usepackage[normalem]{ulem} 

\usepackage{tcolorbox}
\tcbuselibrary{breakable} 
\usepackage{lipsum}
\usepackage{hyperref}

\usepackage{verbatim} 

\title{Graduated  categories of presheaves}

\author{J. Ad\'amek and L. Sousa\footnote{The authors acknowledge financial support by the Centre for Mathematics of the University of Coimbra (CMUC, https://doi.org/10.54499/UID/00324/2025) under the Portuguese Foundation for Science and Technology (FCT), Grants UID/00324/2025 and UID/PRR/00324/2025.}}

\begin{document}

\maketitle

\begin{center}
    \itshape
    Dedicated to Ross Street
\end{center}
\vspace{1.5em}

\begin{abstract}We study conditions on a small category $\cata$ under which  every finitely generated object of its presheaf category is finitely presentable. Example: this holds for a group $\cata$ iff $\cata$ is Noetherian. For ordinals $\cata=\alpha$ this holds iff $\alpha\leq \omega$,  whereas this always holds for $\cata=\alpha^{\op}$.

Various important properties of set functors also apply to endofunctors on
locally finitely presentable categories which are {\em graduated}.  This means that every finitely presentable object $X$ carries a grade (in $\mathbb{N}$) and grades respect subobjects and strong quotients of $X$. We characterize presheaf categories $\Set^{\cata^{\op}}$ which are graduated  for three types of small categories $\cata$. If $\cata$ is a poset, all down sets of elements must be finite. For a group $\cata$, a finite bound on the length of a chain of subgroups must exist. In the case of cartesian categories $\cata$, every object must carry a finite number of sieves. Example: all finitary set functors form a \gr \catc here $\cata$ is the dual of finite sets.

A closely related concept is a locally finitely presentable category with the descending chain  condition (a {\em DCC} category): every finitely presentable object has only finite descending chains of subobjects or strong quotients. The presheaf category on a group $G$ is $DCC$ iff $G$ has no infinite chain of subgroups. Example:  presheaves on the group $\mathbb{Z}$ are not $DCC$, but finite generation implies finite presentation.
\end{abstract}

\section{Introduction}
 We investigate finitarity properties of the categories of presheaves
$\Set^{\mathcal{A}^{\op }}$.

Which categories $\cata$ have the property that every finitely generated \prs is finitely presentable (shortly fg = fp)?
If $\cata$ is a group, then this holds iff $\cata$ is Noetherian
(every ascending chain of subgroups is finite) (Theorem~\ref{thm:4.5}). If $\cata$ has finite products,
then this holds iff every sieve is finitely generated (Corollary~\ref{cor:*3.3}). For general categories,
a sufficient condition for fg = fp is presented in Section 3: every
bi-sieve is finitely generated.

Another property we investigate is graduatedness, since this implies (as we proved in
\cite{AS2024}) that finitary endofunctors on $\Set^{\mathcal{A}^{\op }}$
are right adjoints iff they preserve countable limits:

\begin{definition}
A locally finitely presentable category is \textit{graduated} if every finitely presentable object $X$ carries a grade, a natural number, such that all (proper) subobjects and all (proper)  strong quotients are also finitely presentable, and  have grades (strictly) lower than $X$.
\end{definition}

A number of graduated categories were presented in \cite{AS2024}: sets, graphs, vector spaces, Boolean algebras, etc.

For three classes of small categories $\mathcal{A}$ we give a necessary and sufficient condition for the presheaf category to be graduated:

\begin{enumerate}[(1)]
    \item If $\mathcal{A}$ is a group: there exists a longest finite chain of subgroups.
    Example: presheaves on $\mathbb{Z}$ are not graduated (Section 5).

  Analogously for a groupoid: every automorphism group fulfills the above condition. Example:
     $\text{Set}^{\mathcal{B}^{op}}$ is graduated for the category $\mathcal{B}$ of finite sets and bijections.
    \item If $\mathcal{A}$ is a poset: every element has a finite down-set (Theorem~\ref{thm:6.3}). Example: For an ordinal $\alpha$ the category $\Set^{\alpha^{\op}}$ is graduated iff $\alpha \leq \omega$.
    \item If $\mathcal{A}$ has finite products: every object carries finitely many sieves (Theorem~\ref{thm:new5.1}). Example:
presheaves on the category $\mathcal{F}$ of finite sets and mappings ($\text{Set}^{\mathcal{F}^{op}}$), essentially the category of finitary set functors,  are graduated.
 \end{enumerate}

Another application of graduated categories $\mathcal{K}$ is exhibited in \cite{AMM}: the terminal coalgebra of every finitary endofunctor of $\mathcal{K}$ preserving finite intersections is obtained in countably many steps of the standard construction (see Section 6 for details). Actually, a slightly milder condition on $\mathcal{K}$ is enough for that result.

\begin{definition}
A locally finitely presentable category satisfies the \textit{descending chain condition} (shortly, is a \textit{DCC category}) if for every finitely presentable object $A$ we have that
\begin{enumerate}[a.]
    \item Every subobject is finitely presentable, and all descending chains of subobjects are finite.
    \item Every strong quotient is finitely presentable, and all descending chains of strong quotients are finite (see Convention \ref{con:con1}).
\end{enumerate}
\end{definition}

For example, given a group $G$, the presheaf category $\Set^{G^{\op}}$ is  DCC if and only if every  chain of subgroups is finite (Corollary \ref{cor:4.10}).

\section{Finitely presentable presheaves}

In this preliminary section we  introduce definitions and notations needed in later sections and study conditions under which a \fp \prs is \fgp

Throughout the paper, $\cata$ denotes a small category.

\begin{convention}\label{con:con1} Given an object $X$ in a category $\mathcal{K}$, the ordering of subobjects $m \colon M \rightarrowtail X$ is well known: one follows the direction of arrows in the slice category $\mathcal{K}/X$. The ordering of quotients $e \colon X \twoheadrightarrow Q$ is less standard. In our paper, we follow the direction of arrows of the co-slice category $X/\mathcal{K}$: given $e' \colon X \twoheadrightarrow Q'$, then $e \leq e'$ means that the commutative triangle below exists:

\[
\begin{tikzcd}
& X \arrow[dl, "e"'] \arrow[dr, "e'"] & \\
Q \arrow[rr] & & Q'
\end{tikzcd}
\]

Thus, in $\Set$ the terminal object yields the largest quotient.
\end{convention}

\begin{remark}\label{rem:2.1}
\begin{enumerate}[(1)]
    \item The presheaf category $\Set^{\mathcal{A}^{\op}}$  is equivalent to a variety of $S$-sorted unary algebras \cite{AR}, Ex.~3.5.13, with $S=\text{obj}\, \cata$, and operations indexed by the morphisms of $\cata$: given a
 morphism $f\colon a \to b$ in $\mathcal{A}$, the sorting of the operation is $b$ (input) and $a$ (output).
The variety $\Set^{\mathcal{A}^{\op}}$ is presented by the following equations:
\begin{equation*}\begin{array}{ll}
                  (g ( f(x)) = h(x) & \text{for all }  h = g \circ f \text{ in } \cata; \\
                  \text{id}_a(x) = x & \text{for all sorts } a\\
                  \end{array}
\end{equation*}

\item Consequently, by \cite{AR}, 3.11, a presheaf $F$ is finitely generated iff  it has a finite set of elements $x_i$ belonging to $Fa_i$  $(i=0, \dots, n-1)$ which \emph{generates} $F$:   every element $x \in Fa$ has the form
\(x = Ff(x_i)
\)
for some $i < n$ and  some morphism $f\colon a \to a_i$ in $\mathcal{A}$.

\item
Every finitely generated presheaf $X$ is a quotient of a finite coproduct of  representable presheaves: use Yoneda Lemma on elements  $x_i \in Xa_i$ $(i<n)$ generating $X$.

\item The category of elements $\el F$ of a presheaf $F$ has as objects pairs $(a, x)$, where $a \in \text{obj}\mathcal{A}$ and $x \in Fa$. Morphisms from $(a, x)$ to $(b, y)$ are the morphisms $f: a \to b$ in $\mathcal{A}$ such that $Ff(y) = x$.

    \item In presheaf categories, every quotient is strong. This follows from the fact that $\Set$ satisfies this property, and strong epimorphisms in functor categories are given object-wise.

    \item If $E: \mathcal{A} \to \mathcal{B}$ is an equivalence between small categories, then the presheaf categories $\Set^{\mathcal{A}^{\op}}$ and $\Set^{\mathcal{B}^{\op}}$ are equivalent. Indeed, the functor $\Phi\colon  \Set^{\mathcal{B}^{\op}} \to \Set^{\mathcal{A}^{\op}}$ defined by precomposition with the equivalence,
    \begin{equation*}
        \Phi X = X \cdot E^{\op}\, ,
    \end{equation*}
    is an equivalence functor.
\end{enumerate}
\end{remark}

\begin{proposition}{\em (\cite{Kelly}, Theorem 5.37 and Corollary 5.41)}
The finitely presentable presheaves in $\mathbf{Set}^{\mathcal{A}^{\mathrm{op}}}$ are precisely the finite colimits of representable functors.
\end{proposition}

{
\begin{example}\label{exa:2.4}
The difference between finitely presentable presheaves and finitely generated ones is the difference between
\begin{enumerate}[a.]
 	\item finite colimits of hom-functors, and
 	\item quotients of finite coproducts of hom-functors.
 \end{enumerate}
Here we demonstrate a simple example of a finitely generated presheaf on the ordinal $\omega+1$ that is not finitely presentable. In Section 6 a necessary and sufficient condition for a. and b. to coincide in case of posets is proved (which $\omega$ satisfies but $\omega+1$ does not).
 	Recall that $\omega+1$ is the chain of natural numbers with a largest element $\omega$ added. Presheaves on it (which form the category of chains of sets indexed by
 all ordinals $n \le \omega$ with the dual order) do not fulfill fg = fp. Indeed, the presheaf $T$ with the following values
\[
Tn = \begin{cases}
\{t\} & n < \omega \\
\{0,1\} & n = \omega
\end{cases}
\]
is generated by the elements $0$ and $1$.

However, $T$ is not finitely presentable. To show this, let $X$ be the constant presheaf with value $\{0,1\}$. Consider the $\omega$-chain of presheaves $X_k$ ($k < \omega$) obtained from $X$ by merging, for every $n < k$, the two elements of $X_n$. The connecting maps $\beta_{k,m} \colon X_k \to X_m$ ($m \ge k$) are the obvious quotients. Then $T$ is the colimit of this chain, with the colimit cocone $\gamma_k \colon X_k \to T$ formed by obvious quotients, too. If $T$ were finitely presentable, then the identity morphism $\mathrm{id}_T \colon T \to \mathrm{colim}\, X_k$ would have an essentially unique (!) factorization through $\gamma_k$ for some $k < \omega$. To give a factorizing morphism
\[
\alpha \colon T \to X_k
\; \text{ with $\mathrm{id}_T = \gamma_k \cdot \alpha$} \]
 means precisely to choose $0$ or $1$: when we choose $0$, then all components $\alpha_n$ for $n \ge k$ take $t$ to $0$, analogously for choosing $1$. This contradicts the essential uniqueness: we have two factorizations, say, $\alpha^0$ and $\alpha^1$, which are merged by $\gamma_k$. But no connecting map $\beta_{k,m}$ ($m \ge k$) merges $\alpha^0$ and $\alpha^1$: indeed, we have $X_m = \{0,1\}$, and the $m$-th component of the composite $\beta_{k,m} \cdot \alpha^i$ takes $t$ to $i$ (for $i = 0,1$).
\end{example}
}

\begin{remark}
To give a subobject of $\cata(-,a)$ in $\Set^{\cata^{\op}}$ means precisely  to give a \emph{sieve} on $a$: a collection of morphisms with codomain $a$, closed under precomposition.
\end{remark}

 Below we work with subobjects of a product of presheaves $\cata(-,a)\times \cata(-,a')$:

\begin{definition}
A \emph{bi-sieve} on a pair $(a, a')$ of objects is a collection $S$ of cones with codomain $(a,a')$, closed under precomposition: for every morphism $u: \bar{b} \to b$ we have that
\[\begin{tikzcd}
	& b & \\
	a && {a'}
	\arrow["f"', from=1-2, to=2-1]
	\arrow["{f'}", from=1-2, to=2-3]
\end{tikzcd}\; \;  \in S\qquad \qquad \; \, \text{implies} \qquad \qquad
\begin{tikzcd}
	& {\bar{b}} & \\
	& b \\
	a && {a'}
	\arrow["u", from=1-2, to=2-2]
	\arrow["f"', from=2-2, to=3-1]
	\arrow["{{f'}}", from=2-2, to=3-3]
\end{tikzcd} \,\; \in S.\]
\end{definition}

\begin{example}\label{exa:2.6}
\begin{enumerate}[(1)]
    \item If $\cata$ has binary products, then bi-sieves on $(a, a')$ are precisely sieves on $a \times a'$: to every cone $(f,f')$ corresponds the map $\langle f, f'\rangle$.
    \item A group $\cata$ has only two sieves: $\emptyset$ and $\cata$. Indeed, if $S$ is a sieve with $f_0 \in S$, then for every $f \in \cata$, we have $f = f_0  (f_0^{-1}  f)$, which belongs to $S$. Bi-sieves correspond bijectively to subsets $X \subseteq \cata$: Let $S_X$ denote the bi-sieve of pairs $(f, g)$ with $fg^{-1} \in X$. Then every bi-sieve $S$ is $S_X$ where $X = \{f\in \cata \mid (f, e) \in S \}$ for the unit $e$.
    \item Every sieve $S$ on $a$ can be considered as a bi-sieve on $(a, a)$: it consists of all $(f, f)$ with $f \in S$.
\end{enumerate}
\end{example}

\begin{definition}
A sieve $S$ is \emph{finitely generated} if there is a finite subset $S_0 \subseteq S$ such that $S$ consists of all precomposites of members of $S_0$.

Analogously, a bi-sieve is \emph{finitely generated} if it consists of all precomposites of members of a finite subset.
\end{definition}

\begin{example}\label{exa:2.8} Let $\cata$ be a group. All sieves are finitely generated: this is clear for the empty sieve, and the sieve $\cata$ is generated by the unit $e$.

In contrast, if all bi-sieves are finitely generated, then $\cata$ is finite.
Take the bi-sieve $\cata\times \cata$ and let $M$ be a finite set of its generators.
For every pair $(x, e)$ in $\cata\times \cata$  there exists $(m_1, m_2)$ in $M$ and $g$ in $\cata$ with $(x, e)=(m_1 g, m_2 g)$. Thus $x= m_1 m_2^{-1}$ which proves that there are only finitely many elements $x$.
\end{example}

Thus, a sieve $S$, considered as a subobject of the representable presheaf, is finitely generated if and only if it is a finitely generated object of $\Set^{\cata^{\op}}$.

The collection of all sieves on $a$ forms a poset ordered by inclusion.

\begin{lemma}\label{lem:2.9}
All sieves  on an object $a$ (or bi-sieves on a pair $(a,a')$) are finitely generated iff the poset of all sieves on $a$ (or bi-sieves on  $(a,a')$) contains  no infinite ascending chain.
\end{lemma}

\begin{proof} We present a proof for sieves, that for bi-sieves is analogous.

 (1)  Necessity. Suppose all sieves are finitely generated. Given a set $\{S_i\}$ $(i \in \mathbb{N})$ of sieves on $a$ such that  $S_i  \subseteq S_{i+1}$ for $i \in \mathbb{N}$,
the sieve
$$S = \bigcup_{i \in I} S_i$$
 is generated by a finite subset $M$. Then there is $i$ with $M \subseteq S_i$, thus $S \subseteq S_i$ and $S_k=S_i$ for all $k\geq i$.

 \medskip
(2) Sufficiency. Suppose the poset of all sieves on $a$ has the above property. We prove that every sieve $S$ on $a$ is finitely generated.
 For every  $S\not=\emptyset$ choose  $f_0 \in S$, and let  $S^0$ be the sieve generated by $f_0$ (consisting of all precomposites  $f_0\cdot u$). If $S=S^0$ then it is finitely generated, if not, let  $f_1\in S-S^0$ and let $S^1$ be the  subsieve generated by $f_0$ and $f_1$. This way,  we obtain an ascending chain $S^0\subset S^1\subset \dots$ which, by hypothesis, cannot be infinite. That is, $S$ is generated by $\{f_i\}_{i\le n}$ for some $n$.
\end{proof}

\begin{proposition}\label{pro:fg=fp}
For every small category the following conditions are equivalent:
\begin{enumerate}[(a)]
\item Finitely presentable presheaves are closed under subobjects and quotients.
\item $fg = fp$ holds: every finitely generated presheaf is finitely presentable.
\end{enumerate}
\end{proposition}

\begin{proof} (a) $\Rightarrow$ (b) follows from Remark~\ref{rem:2.1}(3).

(b) $\Rightarrow$ (a).  In view of \ref{rem:2.1}(3),  we need to prove that if $fg = fp$ holds, then for every subobject $m: X \to Y$ with $Y$ a finitely presentable presheaf, $X$ is also finitely presentable.

\medskip
(1)
We first consider a subobject $m: X \to Y$ in $\Set$. We form the amalgam (pushout) on the left
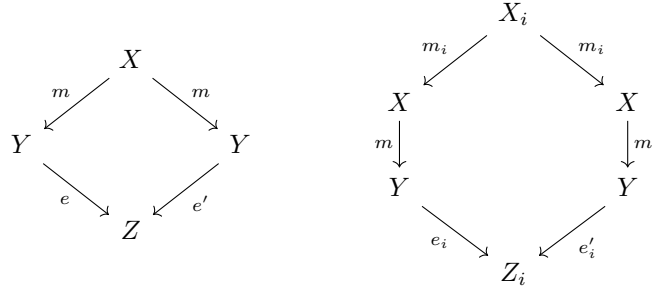
\[\begin{tikzcd}
	& X & \\
	Y && Y \\
	& Z
	\arrow["m"', from=1-2, to=2-1]
	\arrow["{{m}}", from=1-2, to=2-3]
	\arrow["e"', from=2-1, to=3-2]
	\arrow["{e'}", from=2-3, to=3-2]
\end{tikzcd}
\qquad \qquad
\begin{tikzcd}
	& {X_i} & \\
	X && X \\
	Y && Y \\
	& {Z_i}
	\arrow["{m_i}"', from=1-2, to=2-1]
	\arrow["{m_i}", from=1-2, to=2-3]
	\arrow["m"', from=2-1, to=3-1]
	\arrow["m", from=2-3, to=3-3]
	\arrow["{e_i}"', from=3-1, to=4-2]
	\arrow["{e'_i}", from=3-3, to=4-2]
\end{tikzcd}\]
This is also a pullback. Suppose that $X = \bigcup_{i \in I} X_i$ is a directed union of subsets $m_i: X_i \to X$. Form the amalgams on the right. They form  a direct diagram $D$ of quotients:
\(
[e_i, e'_i] : Y + Y \to Z_i
\)
(since $m_i \le m_j$ implies $[e_i, e'_i] \le [e_j, e'_j]$). The colimit of $D$ is given by the canonical morphisms $z_i: Z_i \to Z$ with $z_i \cdot e_i = e$ and $z_i \cdot e'_i = e'$.

Observe that the connecting maps $z_{ji}\colon Z_j\to Z_i$ (for $m_i\subseteq m_j$) are epimorphisms, and if $z_i$ is an isomorphism, then $X_i = X$.

\medskip

(2) Back to subobjects $m: X \to Y$ in $\Set^{\cata^{\op}}$, with $Y$ finitely presentable, we express $X$ as a direct union of all  finitely generated subobjects $m_i: X_i \to X$ $(i \in I)$. We form the amalgams as above (in $\Set^{\cata^{\op}}$). We again obtain a direct diagram $D$ of quotients $Z$ of $Y +Y$ with colimit $z_i: Z_i \to Z$. Moreover, $Z$ is finitely presentable because, as a quotient of $Y + Y$, it is finitely generated. Consequently, the morphism $\id_Z\colon Z\to \text{colim} D$ factorizes essentially uniquely through $z_j$ for some $j\in I$: we have
$$g\colon Z\to Z_j \text{ with } z_j\cdot g=\id_Z.$$
The colimit map $z_j$ coequalizes $g\cdot z_j$ and $\id_Z$, therefore some connecting map $z_{ji}\colon Z_j\to Z_i$ for $i\geq j$ also coequalizes them:
$$z_{ji}=z_{ji}\cdot g\cdot z_j.$$
We conclude that $z_i$ is an isomorphism, inverse to $z_{ji}\cdot g$: first
$$z_i\cdot (z_{ji}\cdot g)=z_j\cdot g=\id.$$
To prove $(z_{ji}\cdot g)\cdot z_i=\id_Z$, use that $z_{ji}$ is epic:
$$(z_{ji}\cdot g\cdot z_i)\cdot z_{ji}=z_{ji}\cdot g\cdot z_j=z_{ji}.$$
We conclude that $X\simeq X_i$ is finitely generated, hence, finitely presentable.
\end{proof}

\begin{example}\label{exa:2.11} If $\cata$ is a group, then fg=fp holds iff $\cata$ is Noetherian: all subgroups are finitely generated (Theorem~\ref{thm:4.5}).
\begin{enumerate}[(1)]
\item The group $\mathbb{Z}$ of integers is Noetherian. The presheaf category is equivalent to the variety of unary algebras on one invertible operation. Thus, every finitely generated algebra in this variety is finitely presentable.

    \item More generally, every finitely generated abelian group is Noetherian (Example~\ref{exa:4.5}).

\item
In contrast, the free group $F_2$ on generators $a, b$ is not Noetherian. Indeed, its commutator subgroup $[F_2, F_2]$ is the free group on the generators
\[
a^n b^m a^{-n} b^{-m} \quad \text{for all non-zero integers } n, m
\]
(\cite[Proposition~I.1.4]{S} applied to $A = B = \mathbb{Z}$, thus $A \ast B = \mathbb{Z} \ast \mathbb{Z} = F_2$ and $R = [F_2, F_2]$).
\end{enumerate}
\end{example}

\section{Sufficient conditions}

We start by proving that if a category $\mathcal{A}$ has finitely generated bi-sieves, then $fg = fp$: every finitely generated presheaf is finitely presentable.
   If, moreover, every pair of objects $a, b$ carries finitely many bi-sieves only, then we prove that $\Set^{\mathcal{A}^{\text{op}}}$ is graduated.

\begin{theorem}\label{thm:3.1}
If every bi-sieve in $\mathcal{A}$ is finitely generated, then all finitely generated presheaves are finitely presentable in $\mathbf{Set}^{\mathcal{A}^{\text{op}}}$.
\end{theorem}

\begin{proof}
(1) Given a presheaf $H = \coprod_{i=1}^{n-1} \mathcal{A}(-, a_i)$, then every subobject $K$ of $H \times H$ is finitely generated. Indeed, $H \times H$ is a coproduct of presheaves $\mathcal{A}(-, a_i) \times \mathcal{A}(-, a_j)$, thus $K$ is a coproduct of subobjects of such presheaves. Each of these subobjects, that is, bi-sieves on $(a_i, a_j)$, is finitely generated. Thus, so is $K$.

\smallskip

(2) Let $F$ be a finitely generated presheaf. We have a quotient $\epsilon \colon H \twoheadrightarrow F$ for some $H$ as in (1). Then $\epsilon$ is the coequalizer of its kernel pair $\varphi_1,\varphi_2 \colon G \rightrightarrows H$. As $G$ is a subobject of $H \times H$ (via $\langle \varphi_1, \varphi_2\rangle$), it is finitely generated.

Therefore, we have a quotient $\bar{\epsilon}\colon \bar{H} \twoheadrightarrow G$ for some coproduct $\bar{H} = \coprod_{j=0}^{m-1} \mathcal{A}(-, b_j)$. We obtain $\epsilon$ as a coequalizer of the pair $\varphi_1\cdot \bar{\epsilon}$ and  $\varphi_2 \cdot \bar{\epsilon}$:
$$ \bar{H} \rightrightarrows H\xrightarrow{\epsilon} F.$$
 Indeed, as $\bar{\epsilon}$ is epic, this follows from $\epsilon=\text{coeq}( \varphi_1,\varphi_2)$.

Since $\bar{H}$ and $H$ are finitely presentable, so is $F$.
\end{proof}

\begin{example}\label{exa:3.2}
The above sufficient condition is, unfortunately, not necessary in general. For example,  $fg = fp$ holds for the group $\mathbb{\mathbb{Z}}$ (Example \ref{exa:2.11}(1)), but its bi-sieve $\mathbb{Z} \times \mathbb{Z}$ is not finitely generated. In fact, let $X \subseteq \mathbb{Z} \times \mathbb{Z}$ be a generating set. Thus, every pair $(a_1, a_2) \in \mathbb{Z} \times \mathbb{Z}$ has the form $(b_1 + c, \, b_2 + c)$ for some $(b_1, b_2) \in X$ and $c \in \mathbb{Z}$. Then $a_1 - a_2 = b_1 - b_2$. Since $a_1 - a_2$ can be an arbitrarily large integer, $X$ cannot be finite.
\end{example}

\begin{remark}\label{rem:3.2}
The following (more pleasing) condition

\centerline{ ``Every sieve in $\cata$ is finitely generated."}
 \noindent is {\em not} sufficient for $fg = fp$. Consider the free group $F_2$ on two generators. It does not fulfill $fg = fp$ (Example \ref{exa:2.11}(3)), but its sieves are \fg (Example \ref{exa:2.8}).
\end{remark}

{\color{black}
\begin{corollary}\label{cor:*3.2}
For every small category the following implications hold:

\medskip

     \centerline{bi-sieves are finitely generated $\Rightarrow$ presheaves fulfil fg=fp $\Rightarrow$ sieves are finitely generated.}
     \end{corollary}

Indeed, the right-hand implication follows from Proposition \ref{pro:fg=fp}, since sieves are precisely the subobjects of representable presheaves.

None of the above implications can be reversed:
For the left-hand one, see Example \ref{exa:3.2}. For the right-hand implication, use the free group on two generators: its sieves are finitely generated, but since it is not Noetherian (Example \ref{exa:2.11}(3)), it does not fulfill fg = fp (Theorem \ref{thm:4.5}).
}

{\color{black}
\begin{corollary}\label{cor:*3.3}
For a small category $\cata$ with binary products, presheaves fulfill fg = fp iff sieves in $\mathcal{A}$ are finitely generated.
\end{corollary}

This follows from Example \ref{exa:2.6}(1) and Corollary \ref{cor:*3.2}.
}

{\color{black}
\begin{example}\label{exa:*3.4}
The category of finitary set functors  on $\Set$ satisfies fg=fp. For $\mathcal{F}$ the category of finite sets, finitary  endofunctors are equivalent to $\Set^{\mathcal{F}}$, presheaves on $\mathcal{F}^{\op}$. Every sieve $S$ on an $n$-element set $a$ in  $\mathcal{F}^{\op}$ is determined by the set $\{\text{ker} f; \, f\in S\}$ of equivalence relations on $a$ (for all $f \colon a\to b$ in $S$). Thus, there are only finitely many sieves on $a$.
\end{example}
}

\begin{corollary}\label{cor:3.8} If $\cata$ is a poset, then fg=fp holds for presheaves iff sieves in $\cata$ are finitely generated.
\end{corollary}

Indeed, whenever sieves are finitely generated, then so are bi-sieves.
A sieve on $a$ is a down-closed subset of the down-set $\down a$, and a bi-sieve $S$ on $(a,a')$ is a down-closed subset of the intersection
$\down a$ and $\down a'$. Thus, $S$ is a sieve on $a$, and a set of generators of
that sieve also generates the bi-sieve $S$.

In Section 6 we prove a nicer condition for fg=fp: the poset contains no bounded infinite  ascending chain or anti-chain (Theorem~\ref{thm:6.1}).

{\color{black}
\begin{theorem}\label{thm:*3.5}
Let $\mathcal{A}$ be a small category such that for every pair of objects the poset of all bi-sieves contains no infinite chain. Then the presheaf category is DCC.
\end{theorem}

\begin{proof}
\begin{enumerate}
    \item[a.] Every bi-sieve on a pair $(a, a')$ is finitely generated (Lemma~\ref{lem:2.9}), thus $\text{fg} = \text{fp}$ holds for presheaves (Theorem~\ref{thm:3.1}).

    \item[b.] Let $G$ be a finitely presentable presheaf. To prove that every descending chain of subobjects is finite, consider first $G = \mathcal{A}(-, a)$. Since the poset of subobjects of $G$ is isomorphic to that of all sieves on $a$, and this is a subposet of all bi-sieves on $(a, a)$, we see that $G$ has no infinite descending chain of subobjects. And the same holds for finite coproducts of representables.

    In general, $G$ is a quotient of a finite coproduct of representables. But in $\Set$, then also in $\Set^{\cata^{\op}}$, the pullback functor along an epimorphism $H\twoheadrightarrow G$ determines an injective map from the poset of subobjects of $G$ into the poset of subobjects of $H$. Then $G$ has the desired property.

    \item[c.] Analogously, $G$ has no infinite descending chain of quotients. To prove this, first observe that for every coproduct
    $$G = \coprod_{i < n} \cata(-, a_i)$$
     the presheaf $G \times G$ does not have any infinite descending chain of subobjects. Indeed, we have
$$G \times G = \coprod_{i, j < n} \cata(-, a_i) \times \cata(-, a_j).$$

Every subobject $H$ of $G \times G$ is a coproduct $H = \coprod_{i, j < n} H_{ij}$ of subobjects $H_{ij}$ of $\cata(-, a_i) \times \cata(-, a_j)$. That is, a coproduct  of $n^2$ bi-sieves.  Since no bi-sieve has an infinite descending chain of subobjects, the same is true about chains of subobjects of $G \times G$.

We are ready to prove that $G$ above does not have any infinite descending chain of quotients. It then follows that every finitely generated presheaf (a quotient of some $G$ as above) shares that property.

Let $\epsilon_n \colon G \to F_n\;\,  (n < \omega)$ be quotients with $\epsilon_{n+1} \le \epsilon_n$ for every $n$. Form the kernel  pairs
\[\begin{tikzcd}
	{H_n} & G & {F_n}
	\arrow["{\psi_n}", shift left=2, from=1-1, to=1-2]
	\arrow["{\varphi_n}"', shift right=2, from=1-1, to=1-2]
	\arrow["{\epsilon_n}", from=1-2, to=1-3]
\end{tikzcd}\quad (n < \omega).\]
Then $\langle \varphi_n, \psi_n \rangle \colon H_n \to G \times G$ is a subobject, and from $\epsilon_{n+1} \le \epsilon_n$, it follows that
$$\langle \varphi_{n+1}, \psi_{n+1} \rangle \le \langle \varphi_n, \psi_n \rangle$$
But all these subobjects $\langle \varphi_n, \psi_n \rangle$ cannot form an infinite descending chain. Consequently, there exists $k < \omega$ such that $\langle \varphi_k, \psi_k \rangle$ represents the same subobject as $\langle \varphi_{k+1}, \psi_{k+1} \rangle$.
This implies that $\epsilon_k = \text{coeq}(\varphi_k, \psi_k)$ represents the same quotient as $\epsilon_{k+1}$. Thus, $G$ has no infinite descending chain of quotients.
\qedhere
\end{enumerate}
\end{proof}

}

{\color{black}
\begin{corollary}\label{cor:*3.6}
For every small category the following implications hold:
\[
\text{no infinite chains of bi-sieves} \implies \text{DCC presheaf category} \implies \text{no infinite chains of sieves}.
\]
\end{corollary}

Indeed, the right-hand implication follows from sieves on $a$ being subobjects of the presheaf $\cata(-,a)$. Thus there are no infinite descending chains, whereas ascending chains are finite by Lemma~\ref{lem:2.9}.

None of the above implications can be reversed: in Exampl~\ref{exa:Tarski} a group $A$ is presented with DCC presheaves which is infinite (thus bi-sieves, i.e., subsets of $A$ forming infinite chains). Whereas the group $\mathbb{Z}$ of integers has only two sieves, but does not yield a DCC presheaf category.

\begin{corollary}\label{cor:*3.7}
Let $A$ be a small category with binary products. Presheaves form a DCC category iff sieves on any object do not form infinite chains.
\end{corollary}

A better criterion is presented in Theorem~\ref{thm:new5.1}.

\begin{theorem}\label{thm:3.4}
Let $\mathcal{A}$ be a small category such that every pair of objects carries finitely many bi-sieves. Then $\Set^{\mathcal{A}^{\text{op}}}$ is graduated.
\end{theorem}

\begin{proof}
\begin{enumerate}[(1)]
\item By Theorem \ref{thm:3.1}, $fg = fp$ holds. Indeed, given a bi-sieve $S$ on $(a,a')$, express it as a directed union of all finitely generated bi-sieves $S_i\subseteq S\; (i\in I)$. Then $I$ is both directed and finite. Hence, $S=S_i$ for some $i$.

    \item We first prove that every finitely presentable presheaf $G$ has finitely many subobjects. This holds if $G = \mathcal{A}(-, a)$, since subobjects correspond to sieves on $a$. Thus, this also holds for $G = \coprod_{i=0}^n \mathcal{A}(-, a_i)$ because a subobject is precisely an $n$-tuple of subobjects of $\mathcal{A}(-, a_i)$ for $i<n$.

    To prove that this holds in general, we can assume that a quotient
    $$\epsilon : \coprod_{i=0}^n \mathcal{A}(-, a_i) \twoheadrightarrow G$$
     is given (as $G$ is finitely generated). For every subobject $\mu\colon F\to G$,  we form the pullback of $\epsilon$ and $\mu$. In $\Set$, the pullback functor along a surjection $e\colon X\twoheadrightarrow Y$ determines an injective map from the set of all subobjects of $Y$ into the set of all subobjects of $X$;
   thus, the same holds in $\Set^{\mathcal{A}^{\text{op}}}$. Since $\coprod_{i=0}^n \mathcal{A}(-, a_i)$ has finitely many subobjects, so does $G$.

    \item Conversely, a presheaf $G$ with finitely many subobjects is finitely presentable. Indeed, express $G$ as a directed colimit $G = \text{colim}_{j \in J} G_j$ of the diagram of all finitely generated subobjects.

Since $G$ has only finitely many subobjects and $I$ is directed, $G=G_{j}$ for some $j$. By (1), $G$ is finitely presentable.

\item Every finitely presentable presheaf $G$ has only finitely many quotients. To prove this, use that  $G$ is a quotient of some $\coprod_{i=0}^{n-1} \mathcal{A}(-, a_i)$. Thus, arguing as in  (2), we only need to treat the case
$$G = \coprod_{i=0}^{n-1} \mathcal{A}(-, a_i)\, .$$

Given a quotient $\epsilon\colon G \twoheadrightarrow F$, form its kernel equivalence
\[\begin{tikzcd}
	H && {\coprod_{i=0}^{n-1} \mathcal{A}(-, a_i)} && F
	\arrow["{h'}"', shift right=2, from=1-1, to=1-3]
	\arrow["h", shift left=3, from=1-1, to=1-3]
	\arrow["\epsilon", two heads, from=1-3, to=1-5]
\end{tikzcd}.\]
Then $\epsilon$  is the coequalizer of $h$ and $h'$. Thus, we need to verify that there are only finitely many possible kernel pairs (up to isomorphism). Indeed,  we have a subobject
$$\langle h, h' \rangle : H \longrightarrow \coprod_{i,j=0}^{n-1} \mathcal{A}(-, a_i) \times \mathcal{A}(-, a_j).$$

Now observe that subobjects of $\cata(-,a_i)\times \cata(-,a_j)$ bijectively correspond to bi-sieves on $(a_i,a_j)$.

\item We are ready to prove that $\Set^{\cata^{\op}}$ is graduated. By Item (1) and Proposition \ref{pro:fg=fp}, we just need to define the grades. For every finitely presentable presheaf $X$ let $s(X)$ and $q(X)$ be the number of subobjects and quotients, respectively, and put
\[\text{grade} \, X=s(X)+q(X)\, .\]
If $m \colon Y \hookrightarrow X$ is a proper suboject of $X$, then  $s(Y) < s(X)$, and we verify $q(Y) \leq q(X)$. For that, observe that in $\Set$, thus also in $\Set^{\mathcal{A}^{\op}}$, we have the dual of the property mentioned in Item (2): the pushout of a quotient $e\colon Y\twoheadrightarrow Q$ along a monomorphism $m\colon Y\hookrightarrow X$  is a quotient and the pushout functor induced by $m$ gives an injective map from the set of quotients of  $Y$ into the set of the quotients of $X$. Thus $q(Y)\leq q(X)$.

 Finally, if $\epsilon \colon X \twoheadrightarrow Y$ is a proper quotient, then $q(Y) < q(X)$. Again, we just need to verify that  $s(Y) \leq s(X)$ by forming pullbacks of  monomorphisms $\mu \colon Y' \hookrightarrow Y$ along $\epsilon$. \qedhere
\end{enumerate}
\end{proof}

\begin{corollary}\label{cor:*3.9}
For every small category $\cata$ the following implications hold:
\[
\begin{array}{c}
\text{finitely many bi-sieves on every pair of objects} \\
\Downarrow \\
\text{graduated presheaf category} \\
\Downarrow \\
\text{finitely many sieves on every object}
\end{array}
\]
\end{corollary}

To verify the second implication, let $S_n$ ($n \in \mathbb{N}$) be a collection of sieves on $a$. Let $k$ be the grade of $\cata(-,a)$. The following chain of subobjects of $\cata(-,a)$
\[
S_0, \, S_0 \cup S_1, \, S_0 \cup S_1 \cup S_2, \dots
\]
has length at most $k$. Thus the sieves $S_n$ are not pairwise distinct.

None of the implications above can be reversed: Use the group $\mathbb{Z}$ again for the second implication. For the first one take the infinite group with graduated presheaf category in Example~\ref{exa:Tarski}(2).

}

\begin{corollary}\label{cor:3.13} If $\cata$ is a small category with binary products, then presheaves are graduated iff every object carries finitely many sieves.

\end{corollary}


\section{Presheaves on groups and groupoids}

Presheaves on a group $G$ are proved to satisfy
\begin{enumerate}
    \item $fg = fp$ iff $G$ is \emph{Noetherian}: every ascending chain of subgroups is finite. Equivalently, all subgroups of $G$ are finitely generated.
    \item DCC iff every chain of subgroups is finite.
    \item Graduation iff a finite upper bound on the length of a chain of subgroups exists.
\end{enumerate}

Results about presheaves on groupoids are proved to be analogous.

\begin{notation}
We denote by
\[
G\text{-}\Set
\]
the category of sets $X$ equipped with a right action of $G$. That is, a map $X \times G \to X$, $(x,g) \mapsto xg$, is given satisfying $xe = x$  and $(xg)g' = x(gg')$ for all $x \in X$ and $g, g' \in G$.
 Morphisms $f: X \to Y$ are the equivariant maps:
\[
f(xg) = f(x)g \quad \text{for all } x \in X \text{ and } g \in G.
\]
This category is equivalent to the presheaf category $\Set^{G^{\op}}$: every presheaf $F\colon G^{\op}\to \Set$ assigns to the unique object a set $X$, and to every element $g \in G$ an automorphism of $X$ (denoted by $x\mapsto xg$). Natural transformations then correspond to equivariant maps.
\end{notation}

An important example of a $G$-set is $G$ itself (the unique hom-functor from $G^{\op}$).

\begin{lemma}\label{lem:4.3}
The poset of all quotients of the $G$-set $G$ is isomorphic to the poset of all subgroups of $G$.
\end{lemma}

\begin{proof}
\begin{enumerate}[(1)]
    \item If $H \subseteq G$ is a subgroup, we define a quotient
    \[
    q : G \to G/H
    \]
    as follows. The elements of $G/H$ are the subsets $Hg = \{hg \mid h \in H\}$ for all $g \in G$. The action of $G$ on $G/H$ is by
    \[
    (Hg)g_0 = H(gg_0) \quad \text{for all } g_0 \in G.
    \]
    The quotient map $q$ takes $g$ to $Hg$.

    \item Given a quotient map $f : G \to X$ in $G\text{-}\Set$, we obtain the subgroup $H \subseteq G$ of all $h \in G$ with $f(h) = f(e)$. The corresponding map $q : G \to G/H$ represents the same quotient as $f$. Indeed, we have the equivariant map $i : G/H \to X$ given by
    \[
    i(Hg) = f(g) \quad (g \in G).
    \]
    This is  clearly well-defined. Moreover, it is an isomorphism  because it is surjective (since $f$ is) and injective: $f(g_1) = f(g_2)$ implies $f(g_1 g_2^{-1}) = f(e)$, thus $Hg_1 = Hg_2$. And we have the following commutative triangle:
   \[\begin{tikzcd}
	& G & \\
	{G/H} && X
	\arrow[""{name=0, anchor=center, inner sep=0}, "q"', from=1-2, to=2-1]
	\arrow[""{name=1, anchor=center, inner sep=0}, "f", from=1-2, to=2-3]
	\arrow["i"', from=2-1, to=2-3]
	\arrow["\sim"', shift right=3, draw=none, from=0, to=1]
\end{tikzcd}\]

\item The mapping
\(
H \mapsto G/H
\)
is an isomorphism from the poset of subgroups of $G$ to the poset of quotients of $G$. Indeed, given subgroups $H$ and $H'$, then for the quotients $q \colon G \to G/H$ and $q' \colon G \to G/H'$ we have (following Convention \ref{con:con1}) that
\[
H \subseteq H' \iff q \le q'\, .
\]
From $H \subseteq H'$ we get a canonical quotient map $\bar{q}\colon G/H \to G/H'$ given by $\bar{q}(Hg) = H'g$, and the triangle below commutes:
\[\begin{tikzcd}
	& G & \\
	{G/H} && {G/H'}
	\arrow["q"', from=1-2, to=2-1]
	\arrow["{q'}", from=1-2, to=2-3]
	\arrow["{\bar{q}}"', from=2-1, to=2-3]
\end{tikzcd}\]
Thus, $q \le q'$. Conversely, let $\bar{q}$ be an equivariant map making the above diagram commutative. Then $H \subseteq H'$ since  $\bar{q}(H) = H'$.   \qedhere
\end{enumerate}
\end{proof}

\begin{lemma}\label{lem:4.4}
Every finitely generated $G$-set is a coproduct of finitely many quotients of the $G$-set $G$.
\end{lemma}
\begin{proof}
Let $x_0, \dots, x_{k-1}$ be  a minimal
set of generators of a $G$-set $X$. For every $i<k$ the set
$$H_i=\{g\in G; \, x_ig=x_i\}$$
is a subgroup of $G$: $g\in H_i$ implies $g^{-1}\in H_i$, and, given $g,h\in G$,
$x_i(gh)=(x_ig)h=x_ih=x_i.$
We have a mapping
$$u\colon \coprod_{i<k}G/H_i\to X$$
taking the class $H_ig$ of $g\in G$ modulo $H_i$ to
$$u(H_ig)=x_ig.$$
This is well defined: $H_ig=H_ig'$ implies $g(g')^{-1}\in H_i$, that is $x_ig(g')^{-1}=x_i$, hence $x_ig=x_ig'$. And it clearly is equivariant. Since $\{x_i\}$ is a set of generators, $u$ is surjective.

To prove that $u$ is an isomorphism, we first observe that its restriction to each $G/H_i$ is monic: if $x_ig=x_ih$, then $gh^{-1}\in H_i$, thus $H_ig=H_ih$. Therefore, we just need to prove for $i\not= j$  that, given $H_ig\in G/H_i$ and $H_jh\in G/H_j$, we have $x_ig\not=x_jh$. Assuming the contrary, we show that $x_i$ can be deleted from the set of generators of $X$  --- this contradicts the minimality of $k$. Indeed, from $x_ig=x_jh$ we get $x_i=x_j(hg^{-1})$. Thus every element of $X$ of the form $x_ik$, $k\in G$, has also the form $x_jk'$ for $k'=hg^{-1}k$.
\end{proof}

\begin{theorem}\label{thm:4.5}
For every group $G$ the following conditions are equivalent:
\begin{enumerate}[(1)]
\item fg=fp: every finitely generated $G$-set is finitely presentable.
\item $G$ is Noetherian.
\end{enumerate}
\end{theorem}

\begin{proof}
(1) Necessity.  Let $G$ fail to be Noetherian, i.e., $G$ has an infinite ascending chain of subgroups $H_i$ ($i < \omega$). Then we prove that there is a finitely generated $G$-set that is not finitely presentable. Namely:
$$G/H\, \text{ where } \, H = \bigcup_{i < \omega} H_i\,.$$

This $G$-set is generated by  $H$. Assuming that $G/H$ is finitely presentable, we get a contradiction. Form the $\omega$-chain of canonical quotients
\[
G/H_0 \xrightarrow{\quad q_{0} \quad} G/H_1 \xrightarrow{\quad q_{1} \quad} G/H_2 \longrightarrow \cdots
\]
where $q_{i}(H_i g) = H_{i+1} g$. It is easy to see that its colimit cocone is formed by the canonical quotients $\bar{q}_i : G/H_i \to G/H$. Since $G/H$ is finitely presentable, $\id : G/H \to \mathrm{colim}_{i < \omega} G/H_i$ factorizes essentially uniquely through $\bar{q}_j$ for some $j < \omega$: we have
$$f\colon G/H \to G/H_j\; \text{  with }\, \mathrm{id}=\bar{q}_j \cdot f.$$
Let $g_0 \in G$ be the element with
$$f(H) = H_j g_0.$$
 Then the colimit map $\bar{q}_j$ merges $H_j$ and $H_j g_0$:
\[
\bar{q}_j(H_j g_0) = \bar{q}_jf(H)= H = \bar{q}_j(H_j).
\]
Thus, for some $i \ge j$ the connecting map
$q_{ji} : G/H_j \to G/H_i$ given by $H_j g \mapsto H_i g$ merges $H_j$ and $H_j g_0$:
\[
H_i = H_i g_0.
\]
We conclude that $\bar{q}_i : G/H_i \to G/H$ is an isomorphism, inverse to $q_{ij}\cdot f$. That is, $H = H_i$. This is the desired contradiction. Indeed, we have:
\[
\bar{q}_i \cdot (q_{ji} \cdot f) = \bar{q}_j \cdot f = \mathrm{id}.
\]
To see that $(q_{ji} \cdot f) \cdot \bar{q}_i = \mathrm{id}$, it is sufficient to verify $q_{ji}\cdot f\cdot \bar{q}_i(H_i)=H_i$:
\[
q_{ji} \cdot f \cdot \bar{q}_i(H_i) = q_{ji} \cdot f(H) = q_{ji}(H_j g_0) = H_i g_0 = H_i.
\]

\medskip
(2) Sufficiency. If $G$ is Noetherian, we prove that $G/H$ is finitely presentable for every subgroup $H$. Then Lemma \ref{lem:4.4} implies that all finitely generated $G$-sets are finitely presentable.

 For every $G$-set $C$, to give an equivariant map $f : G/H \to C$ means to give an element
 $$x \in C\; \text{  with } \;x h = x\, \text{ for all }h \in H.$$
Indeed, given such an $x$, define $f$ by $f(Hg)=xg$ for all $g\in G$.  Conversely, given an equivariant map $f : G/H \to C$, then $x = f(H)$ fulfills
\[
x h = f(H) h = f(Hh) = f(H) = x.
\]

Let $D$ be a directed diagram of $G$-sets with a colimit cocone $c_i\colon D_i\to C\; (i\in I)$:
\[\begin{tikzcd}
	& {D_j} & {D_i} \\
	{G/H} & C
	\arrow["{D(j\to i)}", from=1-2, to=1-3]
	\arrow["{c_j}", from=1-2, to=2-2]
	\arrow["{c_i}", from=1-3, to=2-2]
	\arrow["{\bar{f}}", from=2-1, to=1-2]
	\arrow["f"', from=2-1, to=2-2]
\end{tikzcd}\qquad (i<j\,\text{ in }\,I)\]
For every equivariant map $f : G/H \to C$ we prove that $f$ factorizes, essentially uniquely, through some $c_j$.

Every Noetherian group has all subgroups finitely generated. Let $h_0, \dots, h_{n-1}$ be generators of $H$. Put
$$y = f(H) \in C.$$
Then $f(Hg) = y g$ for all $Hg \in G/H$.

The element $y$ has the form $y = c_j(x)$ for some $j \in I$ and $x \in D_j$. Then $c_j$ merges $x h_p$ with $x$:
\[
c_j(x h_p) = c_j(x) h_p = y h_p = y = c_j(x) \quad (p < n).
\]
Since $D$ is directed, there exists $i \ge j$ such that $D(j \to i)$  also merges $x h_p$ and $x$. That is, for
$$\bar{x} = D(j \to i)(x)$$
we have:
\[
\bar{x} h_p = D(j \to i)(x h_p) = D(j \to i)(x) = \bar{x} \; \; \text{ for all $p<n$.}
\]
Since $H$ is generated by $h_0,\dots, h_{n-1}$, this implies
$$\bar{x} h = \bar{x}\; \text{ for all } \;h \in H.$$

The corresponding equivariant map
$$\bar{\bar{f}}: G/H \to D_i, \; \bar{\bar{f}}(Hg) = \bar{x} g$$
gives the desired factorization:
$
f(Hg) = c_i \cdot \bar{\bar{f}}(Hg),
$
for all $g\in G$. Indeed, from $c_j=c_i\cdot D(i\to j)$ we get
$c_i \cdot \bar{\bar{f}}(Hg) = c_i(\bar{x} g) = c_j(x g) = y g = f(Hg)$.

Finally, $\bar{\bar{f}}$ is essentially unique: if $\tilde{f} : G/H \to D_i$ fulfills $c_i \cdot \tilde{f} = f$, there exists $i' \ge i$ such that $D(i \to i') \cdot \bar{\bar{f}} = D(i \to i') \cdot \tilde{f}$. Indeed, put $\tilde{y} = \tilde{f}(H)$, then $c_i(\tilde{y}) = f(H) = c_i(y)$. Thus there is $i' \ge i$ such that $y$ and $\tilde{y}$ are merged by $D(i\to i')$. This implies the desired equality.
\end{proof}

\begin{example}\label{exa:4.5}
\begin{enumerate}[(1)]
\item For an abelian group $G$, presheaves fulfil fg=fp iff $G$ is finitely presentable (=Noetherian). Indeed, every \fg abelian group $G$ is Noetherian. This follows from the fact that $G$ is a finite coproduct of groups $\mathbb{Z}$ or $\mathbb{Z}/n\mathbb{Z}$, each of which is \fpp Thus all subgroups of $G$ are \fpp
\item
{\color{black}
A concrete example of a presheaf on a group that is finitely generated but not finitely presentable is provided by the quotient of the free group $F_2$ on generators $a$ and $b$ (Example~\ref{exa:2.11}(3)) modulo its commutator $[F_2,F_2]$.
This quotient is the group $\mathbb{Z} \times \mathbb{Z}$, viewed as the following presheaf:
\[
\mathbb{Z} \times \mathbb{Z} = F_2 / [F_2,F_2].
\]
 To see that $\mathbb{Z} \times \mathbb{Z}$ is not a finitely presentable $F_2$-set,  we use the fact that $[F_2,F_2]$ is the free group on the set
$$a^{-n}b^{-m}a^nb^m \quad (n,m\in \mathbb{Z}-\{0\})$$
as seen in Example~\ref{exa:2.11}(3).

Express the set of generators of $[F_2, F_2]$ as a union of an $\omega$-chain of finite subsets, and let $H_k$ be the free group on the $k$-th subset. Then
\[
[F_2, F_2] = \bigcup_{k < \omega} H_k.
\]
Thus, $\mathbb{Z} \times \mathbb{Z}$ is the colimit of the corresponding $\omega$-chain $F_2/H_k$ of $F_2$-sets.
Arguing as in the proof of Theorem \ref{thm:4.5}, this proves that $\mathbb{Z} \times \mathbb{Z}$ is not a finitely presentable $F_2$-set: $\mathrm{id}_{\mathbb{Z} \times \mathbb{Z}}$ does not factorize essentially uniquely through any colimit map.
}
\end{enumerate}
\end{example}

\begin{theorem}\label{thm:4.7}
The category $G$-$\Set$ is DCC iff $G$ has no infinite chain of subgroups.
\end{theorem}

\begin{proof}
If the category $G$-$\Set$ is DCC, then $G$ has no descending infinite chain of subgroups by Lemma \ref{lem:4.3}, and no ascending one  by Theorem \ref{thm:4.5} and Proposition~\ref{pro:fg=fp}.

Conversely, let $G$ have no infinite chain of subgroups. Then by Lemma  \ref{lem:4.4}  we need to prove that each $G$-set $\displaystyle{\coprod_{i=0}^{n-1}G/H_i}$  has only finite chains of subobjects or quotients.

\begin{enumerate}[(a)]
    \item Subobjects. We first verify that, given a subgroup $H \subseteq G$, then $G/H$ has no non-trivial subobjects. Indeed, given a nonempty subobject $S \subseteq G/H$, it is all of $G/H$: choose $H g_0 \in S$. For each $g \in G$, we have:
    $
    H g = H g_0 (g_0^{-1} g) \in S$.
    Thus, the coproduct above has precisely $2^n$ subobjects: every subobject is a coproduct $\displaystyle{\coprod_{i=0}^{n-1}K_i}$ with $K_i$ either $\emptyset$ or $G/H_i$.

    \item Quotients. Given a quotient $r : \coprod_{i=0}^{k-1} G/H_i \to Y$, we observe that, whenever $r$ merges elements $H_i g_0$ and $H_j g_0'$ for some $i \ne j$ and $g_0, g_0' \in G$, then $r[G/H_i] = r[G/H_j]$. Indeed, given $H_i g \in G/H_i$, we have:
    \[
    r(H_i g) = r(H_i g_0 g_0^{-1} g) = r(H_j g_0') g_0^{-1} g \in r[G/H_j].
    \]
    Thus, we just need to prove that for every subgroup $H$, the quotients of $G/H$ do not form infinite descending chains. This follows from Lemma  \ref{lem:4.3}: descending chains of quotients of $G/H$ correspond to descending chains of subgroups of $H$. \qedhere
\end{enumerate}
\end{proof}

{\color{black}
\begin{example}\label{exa:Tarski} \begin{enumerate}[(1)]
\item  Presheaves on the group $\mathbb{Z}$ of integers are not DCC: consider the chain of subgroups $k\mathbb{Z}$ for $k=2,4,8,\dots$.
Consequently, presheaves on an abelian group $G$ are DCC iff $G$ is finite.
Indeed such a group is a finite coproduct of groups $\mathbb{Z}/n\mathbb{Z}$.

\item Infinite groups with DCC categories of presheaves are rare.  But they exist; for example, the \emph{Tarski monster groups}, which are  infinite groups such that, for a given prime $p$, every proper subgroup has $p$ elements. The existence of such a group for every prime $p>10^{75}$ is proved in \cite{Olshanskii1980}.
    \end{enumerate}
\end{example}
}

\begin{theorem}\label{thm:4.9}
For a group $G$, the category of $G$-sets is graduated iff a longest finite chain of subgroups exists.
\end{theorem}

\begin{proof}
Necessity follows from Lemma \ref{lem:4.3}: if the $G$-set $G$ has grade $n$, then no chain of proper subgroups has length larger than $n$.

To prove sufficiency, let $m$ be the maximum length of a chain of subgroups of $G$. We first define the grade of $G/H$ for a subgroup $H$ of $G$. Observe that $G/H$ has no nontrivial subobjects in $\Set^{G^{\op}}$, see item (a) of the proof of Theorem \ref{thm:4.7}.
 Quotients of $G/H$ are precisely all $G/K$ for a subgroup $K\subseteq G$ containing $H$. We define
\[
\text{grade}\, G/H=\text{length of a  longest ascending chain of subgroups of $G$ starting with $H$.}
\]
This is at most $m$.

For a  finitely presentable $G$-set $X=\coprod_{i<n} G/H_i$ (see Lemma \ref{lem:4.4})  we put
\[
\text{grade} \, X = \sum_{i<n} \text{grade}\,  G/H_i.
\]

Subobjects of finitely presentable presheaves are finitely presentable, since they are essentially coproducts of quotients of $G$ (see Theorem \ref{thm:4.7}). Every proper subobject $Y$ of $X$ has a smaller grade. Indeed, $Y=\coprod_{i<n}Y_i$ where $Y_i$ is a subobject of $G/H_i$, and for some $i$ then $Y_i$ is the initial presheaf (of values $\emptyset$).

 We just need to prove that if $e\colon X\to Y$ is a proper quotient, then $\text{grade}\, Y< \text{grade}\, X$.  We know that $Y$ has the form $Y=\coprod_{j<p}G/K_j$ (Lemma \ref{lem:4.4}). Then, for every $j<p$, there is some $i_j<n$ such that $H_{i_j}\subseteq K_j$. By restriction and corestriction, $e$ gives rise to a quotient
$$G/H_{i_j}\twoheadrightarrow  G/K_j.$$
Since $e$ is a proper quotient, either $p<n$ or $H_{i_j}\subset K_j$ for some $j$.  Thus $\text{grade}\, Y < \text{grade}\, X$.
\end{proof}

\begin{example}\label{exa:Obraztsov}
There exist groups $G$ with arbitrarily long finite chains of subgroups, but with no infinite chain, see Appendix. Then, $\Set^{G^{\op}}$ is DCC but not graduated.
\end{example}

Recall that a groupoid is a category $\mathcal{G}$ with invertible morphisms. Every object $a$ defines the automorphism group $\mathcal{G}(a,a)$.

\begin{corollary}\label{cor:4.10} The category of preasheaves on a small groupoid satisfies
\begin{enumerate}
\item fg=fp iff every automorphism group is Noetherian.
\item DCC iff every automorphism group has only finite chains of subgroups;
\item Graduatedness iff every automorphism group has a finite longest chain of subgroups.
\end{enumerate}
\end{corollary}

\begin{proof}

Every groupoid $\mathcal{G}$ is a coproduct of its components, the maximum connected subgroupoids:
\[\mathcal{G}=\coprod_{i\in I} \mathcal{G}_i\, .
\]
Every  finitely presentable presheaf \(X\)  on \(\mathcal{G}\) has a finite generating set. It follows that its restriction to all but finitely many summands,
\(
X_i = X/{\mathcal{G}_i},
\)
is constant with value \(\emptyset\).

Consequently, it is enough to prove the corollary for connected groupoids \(\mathcal{G}\). For them this follows from Remark \ref{rem:2.1}, Lemma \ref{lem:4.3},  Theorems \ref{thm:4.5}, \ref{thm:4.7} and  \ref{thm:4.9}, because \(\mathcal{G}\) is a category equivalent to the group
\(
G = \mathcal{G}(a,a)
\)
for any object \(a \in \mathcal{G}\)
(thus,
\(
\mathbf{Set}^{\mathcal{G}^{\mathrm{op}}}
\) is equivalent to
\(\mathbf{Set}^{G^{\mathrm{op}}}\) by Remark~\ref{rem:2.1}(6)).
Indeed, for every object \(b \in \mathcal{G}\) there exists, since \(\mathcal{G}\) is connected, a morphism
\(
i_b : b \to a
\).
Consider the functor
\[
\Phi : \mathcal{G} \to G
\]
 assigning to every morphism
\(
f : b \to b'
\) of $\mathcal{G}$
the morphism
\[
\Phi f = a\xrightarrow{i_b^{-1}}b\xrightarrow{f}b'\xrightarrow{i_{b'}}a .
\]
This functor is fully faithful, and hence an equivalence.
\end{proof}

\begin{example}
Let \(\mathcal{B}\) be the groupoid of finite sets and bijections. The presheaf category
\(
\mathbf{Set}^{\mathcal{B}^{\mathrm{op}}}
\)
is graduated. Indeed, the automorphism groups are finite.
\end{example}

\section{Presheaves on a cartesian category}

From Section 3 we know that if sieves are finitely generated in a cartesian category $\cata$, then fg = fp holds in $\Set^{\cata^{\op}}$. We now prove that  the sufficient   condition for graduatedness (every object carries a finite number of sieves) turns out to be also necessary. In fact, the following nicer condition is necessary and sufficient: every collection of morphisms $t_n \colon a_n \to a$ ($n < \omega$) in $\mathcal{A}$ factorizes through some of a finite number of their members.

\begin{theorem}\label{thm:*5.2} The following conditions on a small category $\cata$ with binary products are equivalent:
\begin{enumerate}[(1)]
    \item $fg = fp$: Every finitely generated presheaf is finitely presentable.
    \item For every collection of morphisms $t_n \colon a_n \to a$ ($n < \omega$) in $\mathcal{A}$ there exists $k$ such that each $t_n$ factorizes through $t_i$ for some $i \leq k$.
    \item All sieves in $\mathcal{A}$ are finitely generated.
    \item All bi-sieves in $\mathcal{A}$ are finitely generated.
\end{enumerate}
\end{theorem}

\begin{proof} We prove the equivalence of (2) and (3), the rest follows from Corollary~\ref{cor:*3.3}.

\medskip

$(2) \Rightarrow (3)$: By Lemma~\ref{lem:2.9} we need to prove that every ascending chain of sieves $S_n$ on an object $a$ is finite. Assuming the contrary, we choose $t_n \in S_{n+1} \setminus S_n$ for every $n < \omega$, and verify that this collection of morphisms fails to satisfy (2). Indeed, for any $i \le k$, $t_{k+1}$ does not factor through $t_i$. This holds because $t_i \in S_{i+1} \subseteq S_{k+1}$, whereas $t_{k+1} \notin S_{k+1}$.

\medskip

$(3) \Rightarrow (2)$: Let $t_n \colon a_n \to a$ ($n < \omega$) be given. Denote by $S_n$ the sieve of all morphisms that factorize through $t_i$ for some $i \leq n$. Then $S_n \subseteq S_{n+1}$. By Lemma~\ref{lem:2.9} there exists $k$ such that $S_k$ is the largest of all these sieves $S_n$. Then $k$ has the desired property: every morphism $t_n$ lies in $S_k$, thus $t_n$ factorizes through $t_i$ for some $i \leq k$.
\end{proof}

\begin{theorem}\label{thm:new5.1}
For every category $\cata$ with binary products, the following conditions are equivalent:
\begin{enumerate}[(1)]
    \item $\Set^{\cata^{\op}}$ is graduated.
    \item $\Set^{\cata^{\op}}$ is DCC.
    \item Every object of $\cata$ carries finitely many sieves.
\end{enumerate}
\end{theorem}

\begin{proof}
$3 \implies 1$ follows from Corollary~\ref{cor:*3.9}, and $1 \implies 2$ is clear. Let us prove $2 \implies 3$.

Given an object $a$, let $L$ be the poset of all sieves on it, ordered by inclusion. This is a complete, distributive lattice with meets given by intersection. This lattice has no infinite descending chains because it is isomorphic to the lattice of subobjects of $\cata(-,a)$ in $\Set^{\cata^{\op}}$. We prove that $L$ is finite.

\begin{enumerate}
    \item Every element $S$ of $L$ is the join of finitely many atoms (join-irreducible elements) of $L$. Indeed, assuming the contrary, we have a sieve $S$ which is not such a join. Thus $S$ is not an atom:
    \[
    S = S_0 \vee S_1 \quad \text{for some $S_i \subset S$ ($i = 0, 1$).}
     \]
     Moreover, $S_0$ or $S_1$ is also not a join of finitely many atoms, say, it is $S_0$. Thus we have
    \[
    S_0 = S_{00} \vee S_{01}
    \quad \text{for some $S_{0i}\subset S_0$.}\]
   We can assume that $S_{00}$ is not a join of finitely many atoms, etc. We obtain an infinite descending chain
    \[
    S \supset S_0 \supset S_{00} \supset S_{000} \supset \dots,
    \]
    a contradiction.

    \item The join in 1. is unique. Indeed, given $S = S_0 \vee S_1 \vee \dots \vee S_{n-1}$ for atoms $S_i$, then every atom $\overline{S} \subseteq S$ is equal to $S_i$ for some $i$. This follows from $\overline{S} \wedge S_i = \emptyset$ or $S_i$, and
    \[
    \overline{S} = \overline{S} \wedge S = \bigvee_{i=0}^{n-1} (\overline{S} \wedge S_i).
    \]

    \item $L$ has finitely many atoms. Indeed, the largest presheaf $\cata/a$ is a join $S_0 \vee \dots \vee S_{n-1}$ of atoms. Due to distributivity, every sieve on $a$ 3is a join of a subset of $\{S_i\}_{i < n}$, thus,
    \(
    \text{card} L \leq 2^n.
    \)
\end{enumerate}
\end{proof}

\begin{example}
\begin{enumerate}
    \item $\Set^{\mathcal{F}}$ is graduated (see Example \ref{exa:*3.4}(1)). Also, $\Set^{\mathcal{F}^{\op}}$ is graduated: a sieve in $\mathcal{F}$ on an $n$-element set $a$ is determined by the set of all images of members of $\mathcal{F}$, a subset of $\mathcal{P}a$. Thus, there are at most $2^{2^n}$ sieves.

    \item The poset $\omega^{\op}$ has infinitely many sieves on $0$: each down-set $\down n$ in $\omega^{\op}$ is a sieve. However, all of them are finitely generated. Thus, the category $$\Set^{\omega}$$ of $\omega$-chains of sets satisfies fg=fp, but it is not graduated.
\end{enumerate}
\end{example}

\section{Presheaves on Posets}

The results on $\Set^{\cata^{\op}}$ for posets $\cata$  are completely analogous to those for cartesian categories. However, we have not found a unified proof for these two cases.

Observe that the hom-functor $\cata(-,a)$ is, essentially, the subposet $\down a$ of all elements majorized by $a$.

\begin{theorem}\label{thm:6.1}
The following conditions on a poset $\cata$ are equivalent:
\begin{enumerate}[(1)]
    \item $fg=fp$: finitely generated presheaves are finitely presentable in $\Set^{\cata^{\op}}$.
    \item Every ascending chain or antichain with an upper bound in $\cata$ is finite.
    \item All sieves in  $\cata$ are finitely generated.
    \item All bi-sieves in  $\cata$ are finitely generated.
\end{enumerate}
\end{theorem}

\begin{proof}
$(1) \implies (2)$: Let $\{b_n\}_{n \in \omega}$ be an infinite subset forming an ascending chain or antichain with an upper bound $a$. We show that the presheaf $F$ given by
\begin{equation*}
    F c = \begin{cases}
        \{0, 1\} & \text{if } c \leq a \text{ and $c\not\leq b_k$ for all $k<\omega$}\\
        \{0\} & \text{if } c \leq b_k \text{ for some $k<n$} \\
        \emptyset & \text{otherwise}
    \end{cases}
\end{equation*}
is not finitely presentable. This is the desired contradiction: $F$ is generated by $0,\,1 \in Fa$.

 We clearly have that
\begin{equation}\label{eq:star6}
m < n  \, \text{ implies }\, b_n \not\leq b_m.
\end{equation}

Let $F_n$ be the following presheaf ($n < \omega$):
\begin{equation*}
    F_n c = \begin{cases}
        \{0, 1\} & \text{if } c \le a \text{ and } c \not\leq b_k \text{ for all } k < n \\
        \{0\} & \text{if } c \le b_k \text{ for some } k < n \\
        \emptyset & \text{otherwise}
    \end{cases}
\end{equation*}
It follows from \eqref{eq:star6} that these presheaves form an $\omega$-chain
\[
F_0 \xrightarrow{\epsilon_0} F_1 \xrightarrow{\epsilon_1} F_2 \xrightarrow{\epsilon_3} \dots
\]
of quotients of $F_0$. Moreover, the cocone of quotients $\bar{\epsilon}_n: F_n \to F$ is a colimit of this chain.

If $F$ were finitely presentable,  the identity map $\id_F: F \to \text{colim}_{n < \omega} F_n$ would factorize through  $\bar{\epsilon}_n$ for some $n$:
\[
\begin{tikzcd}
    F \arrow[r, "\mathrm{id}"] \arrow[dr, dashed, "\psi"'] & F \\
    & F_n \arrow[u, "\bar{\epsilon}_n"']
\end{tikzcd}
\]
Clearly, $\psi_a = \mathrm{id}$. We obtain a contradiction: For the morphism $b_n \to a$, the naturality square does not commute:
\[
\begin{tikzcd}
    F a \arrow[r, "\psi_a"] \arrow[d, "F(b_n \to a)"'] & F_n a \arrow[d, "F_n(b_n \to a)"] \\
    F b_n \arrow[r, "\psi_{b_n}"] & F_n b_n
\end{tikzcd}
\quad =\quad
\begin{tikzcd}
    \{0,1\} \arrow[r, "\mathrm{id}"] \arrow[d] &  \{0,1\} \arrow[d, "\mathrm{id}"] \\
    \{0\} \arrow[r] & \{0,1\}
\end{tikzcd}
\]

\medskip
$(2) \implies (3)$: By (2), no infinite antichain or ascending chain with an upper bound exists in $\cata$.

We show that, then, for every $a\in \cata$, no infinite ascending chain of sieves on $a$ exists. Hence, by  Lemma \ref{lem:2.9} we conclude that all sieves of $\cata$ are finitely generated. Indeed, assume that, on the contrary, we have an ascending chain of sieves on $a$:
$$S_0\subset S_1\subset S_2\subset \dots$$
Let $b_0\in S_0$. Then there is some $b_1\in S_1-S_0$ such that $b_1\not\leq b_0$. Analogously, there is some $b_2$ with $b_2\not\leq b_1$, and so on. That is, we find a set of elements $b_i$, $i<\omega$, majorized by $a$, such that $b_n\not\leq b_m$ for $m<n$. If for every $i<\omega$ there is $j\in \omega$ with $i<j$ and $b_i<b_j$, then we get   an infinite ascending chain. Otherwise, there is some $i_0$ such that for all $j>i_0$ $b_{i_0}\not\leq b_j$; analogously, there is some $i_1\geq i_0+1$ such that for all $j\geq i_1$ we have $b_{i_1}\not\leq b_j$; and so on. This way we get an infinite antichain. This is the desired contradiction.

\medskip
$(3) \implies (4)$: Given elements $a$ and $a'$, every bi-sieve on them is also a sieve on $a$.

\medskip
$(4) \implies (1)$: This is Theorem \ref{thm:3.1}.
\end{proof}

\begin{example}\label{exa:6.2}
For every  ordinal $\alpha$ (the poset of all ordinals  $i < \alpha$) the category of $\alpha$-chains of sets (which are presheaves on $\alpha^{\op}$) fulfils $fg=fp$ because $\alpha^{\op}$  contains no infinite ascending chain or antichain.

In contrast, the category of $\alpha^{\op}$-chains of sets does not fulfil $fg=fp$, unless  $\alpha\leq \omega$. See Example~\ref{exa:2.4}.
\end{example}

\begin{theorem}\label{thm:6.3}
The following conditions on a poset $\cata$ are equivalent:
\begin{enumerate}[(1)]
    \item $\Set^{\cata^{\op}}$ is DCC;
    \item $\Set^{\cata^{\op}}$ is graduated;
    \item Every subset of $\cata$ with an upper bound is finite.
\end{enumerate}
\end{theorem}

\begin{proof}
$(3) \implies (2)$. This follows from Theorem \ref{thm:3.4}: given $a, a'\in \cata$, the number of bi-sieves on them is finite because $\down a$ is finite.

\medskip
$(2) \implies (1)$. This is trivial.

\medskip
$(1) \implies (3)$. For every element $a$ we prove that the down-set $\down a$ is finite (which is equivalent to (3)).  Since DCC implies that subobjects of finitely presentable presheaves are finitely presentable, the presheaf
$$\down a  - \{a\}$$
of all elements $b < a$ is \fpc thus, finitely generated.
Let $b_0,\dots,b_{n-1}$ be  generating elements: we have $b_i<a$ and
\[
\down a  = \bigcup_{i=0}^{n-1} \down b_i \cup \{a\}.
\]
We define a directed graph $T$  on the set $\down a$ of vertices. The neighbors of $a$ are $b_i$:
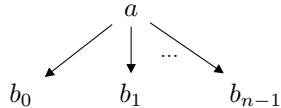
\[
\begin{tikzcd}[
    arrow style=tikz,
    /tikz/blacktriangle/.style={-{Triangle[length=3.5pt, width=3.5pt, fill=black]}}
]
	& a & \\
	{b_0} & {b_1} & {b_{n-1}}
	\arrow[blacktriangle, from=1-2, to=2-1]
	\arrow[""{name=0, anchor=center, inner sep=0}, blacktriangle, from=1-2, to=2-2]
	\arrow[""{name=1, anchor=center, inner sep=0}, blacktriangle, from=1-2, to=2-3]
	\arrow["\dots"'{pos=0.7}, draw=none, from=0, to=1]
\end{tikzcd}
\]
For each $i<n$, we continue analogously: we find generating elements $c_0, \, \dots,c_{k-1}$ of $\down b_i-\{b_i\}$, and they are the neighbors of  $b_i$ in $T$, etc. This yields a directed graph with no infinite directed path: if $x_0, \, x_1, \, x_2, \dots$ were a path in $T$, then $x_0> x_1> x_2>\dots$ and $\down x_n$ ($n<\omega$) would form an infinite descending chain of sieves, in contradiction to the DCC property. Consequently, the creation of $T$ stops after finitely many steps. We conclude that $T$ is a finite graph.

It remains to prove that every element $x \le a$ is a vertex of $T$.
Assuming the contrary, we present an infinite path $y_m$ in $T$.
Put $y_0 = a$. Since $x$ does not lie in $T$, we have $x < a$, therefore, there exists $i < n$ such that $x \le b_i$ for some $i < n$; put $y_1 = b_i$. Analogously, we have $x < b_i$, therefore there exists $j < k$ with $x \le c_j$; put $y_2 = c_j$, etc. Thus the vertices of $T$ are precisely the elements of $\down a$. This concludes the proof that $\down a$ is finite.
\end{proof}

\begin{example}
Let $\alpha$ be an ordinal. Then $\Set^{\alpha}$  is DCC iff $\alpha$ is finite. The same is true for $\Set^{\alpha^{\op}}$.
\end{example}

We have mentioned in the Introduction that DCC categories have the property that finitary functors preserving finite intersections have a simplified construction of terminal coalgebras. We now present the details of this result from \cite{AMM2}, and give an example which demonstrates that the DCC condition is substantial for it.
For the following, see \cite{A} or \cite{AMM1}.

\begin{remark}\textbf{The terminal-coalgebra chain}.
Given an endofunctor $F: \mathcal{K} \to \mathcal{K} $ where $\mathcal{K}$ is complete, a \emph{coalgebra} is an object $A$ of $\mathcal{K} $ together with a morphism $a: A \to FA$. The category of coalgebras has morphisms from $(A,a)$ to $(A',a')$ given by $f: A \to A'$ in $\mathcal{K}$ such that the following square commutes:
\[\begin{tikzcd}
	A & FA \\
	{A'} & {FA'}
	\arrow["a", from=1-1, to=1-2]
	\arrow["f"', from=1-1, to=2-1]
	\arrow["Ff", from=1-2, to=2-2]
	\arrow["{a'}"', from=2-1, to=2-2]
\end{tikzcd}\]
The \emph{terminal coalgebra}, denoted by
$$\nu F \xrightarrow{\psi} F(\nu F)$$
 is the terminal object of the resulting category.
\end{remark}

The following transfinite chain $V_i$ ($i \in \text{Ord}^{\op}$) has the property that whenever a connecting map $v_{i+1,i}: V_{i+1} \to V_i$ is invertible, then $\nu F = V_i$ with $\psi = v_{i,i+1}^{-1}$, as proved in \cite{A}:

\begin{itemize}
    \item {Initial step:} $V_0$ is the terminal object of $\mathcal{K}$.
    \item {Isolated step:} $V_{i+1} = F V_i$.
    \item {Limit step:} $V_i = \lim_{j < i} V_j$.
\end{itemize}

The connecting morphisms $v_{ij}: V_i \to V_j$ ($i > j$) are obvious for $j=0$; we have
\[
v_{i+1, j+1} = F v_{ij}: V_{i+1} \to V_{j+1},
\]
and for limit ordinals $j$, all the maps $v_{ji}: V_j \to V_i$ ($j > i$) form the limit cone.

\begin{definition}
The terminal-coalgebra  chain \emph{converges} in $i$ steps if $v_{i+1, i}: F V_i \to V_i$ is invertible.
\end{definition}

\begin{proposition} (\cite{A})
If the terminal coalgebra chain converges in $i$ steps, then $V_i$ is the terminal coalgebra with structure map $v_{i+1,i}^{-1}\colon V_i\to FV_i$.
\end{proposition}

\begin{theorem} (\cite{AMM2}, Theorem 5.1)
Let $\mathcal{K}$ be a DCC category. Then for every finitary endofunctor $F: \mathcal{K} \to \mathcal{K}$ preserving finite intersections, the terminal-coalgebra chain converges in $\omega + \omega$ steps.
\end{theorem}

We now present an example demonstrating that for a locally finitely presentable category not fulfilling the DCC condition, the convergence can take $\alpha$ steps for an arbitrary large ordinal $\alpha$.
The base category is presheaves on $\alpha^{\op}$.

\begin{example}
We work in the category $\Set^\alpha$ of chains $X = (X_i)_{i<\alpha}$ of sets
\[
X\equiv X_0 \to X_1 \to X_2 \to \dots \to X_i \to \dots \quad (i < \alpha)
\]
Morphisms $\varphi: X \to X'$ are the natural transformations $\varphi_i: X_i \to X'_i$ making the following squares commute for all $i \le j < \alpha$:
\[
\begin{tikzcd}
    X_i \arrow[r] \arrow[d, "\varphi_i"'] & X_j \arrow[d, "\varphi_j"] \\
    X'_i \arrow[r] & X'_j
\end{tikzcd}
\]
We denote by $I$ the initial object with all components $\emptyset$.

We define an endofunctor $F$ on $\Set^\alpha$ by
$F I = I$, and for $X \neq I$ we put
\[
(FX)_i = \begin{cases}
    X_i & \text{if there is } j < i \text{ with } X_j \neq \emptyset \\
    \emptyset & \text{otherwise.}
\end{cases}
\]
That is, if $X$ has a prefix consisting of copies of $\emptyset$, then  $FX$ differs from $X$ by changing just the next component to $\emptyset$. Analogously for morphisms: Let $!$ denote the empty map (domain $\emptyset$), then for $\varphi\colon X\to Y$ we put
\[
(F\varphi)_i = \begin{cases}
    \varphi_i & \text{if there is } j < i \text{ with } X_j \neq \emptyset \\
    ! & \text{otherwise.}
\end{cases}
\]
\end{example}

\begin{lemma}
$F$ is a finitary endofunctor whose terminal-coalgebra chain does not converge before $\alpha$ steps.
\end{lemma}

\begin{proof}
(1) Given a directed diagram $D: \mathcal{D} \to \Set^\alpha$ with a colimit $\gamma^d\colon Dd\to C$ $(d\in \text{obj}\mathcal{D})$, we prove that $\text{colim} \, FD $ is given by $(F\gamma^d)$. Since  colimits in $\Set^\alpha$ are computed object-wise, for every ordinal $i < \alpha$ we have the following equivalence (given any $j<i$):
\[
C_j \neq \emptyset \text{ iff  there exists } d \in \mathcal{D} \text{ with } (Dd)_j \neq \emptyset.
\]
From this, it is easy to see that $F$ preserves the colimit.

\medskip
(2) The functor $F$ has only one coalgebra, namely $\id: I \to FI$. Indeed, given $X \neq I$, there exists  $i < \alpha$ with $X_i \neq \emptyset$ and $(FX)_i = \emptyset$. Thus no morphism leads from $X$ to $FX$. Therefore, $I$ is the terminal coalgebra.

\medskip
(3)
The terminal-coalgebra chain starts as follows:
\begin{align*}
    V_0 : 1 \to 1 \to 1 \to \dots \\
    V_1 : \emptyset \to 1 \to 1 \to \dots \\
    V_2 : \emptyset \to \emptyset \to 1 \to \dots
\end{align*}
In general,

\[
(V_k)_i  = \begin{cases}
    \emptyset & \text{ if  } i < k \\
    1 & \text{otherwise.}
\end{cases}
\]
 It takes $\alpha$ steps to reach $V_\alpha = I$.
\end{proof}

\begin{openproblem}
Can an analogous example be found with $F$ that, moreover, preserves finite intersections?
\end{openproblem}

\appendix

\section{Appendix: Chains of Subgroups}

In Example \ref{exa:Obraztsov} we claim that there  exist groups without infinite chains of subgroups which have arbitrarily long finite chains of subgroups. We have asked  Artificial Intelligence  and obtained an answer that we checked and  present below.

For the quotients $\mathbb{Z}/n$ of the group of integers modulo the subgroup $n\mathbb{Z}$ it is easy to see that $\mathbb{Z}/3^i$ is a subgroup of $\mathbb{Z}/3^{i+1}$ for every $i \in \mathbb{N}$. Thus, $\mathbb{Z}/3^i$ has a chain of subgroups of length $i+1$. No element of $\mathbb{Z}/3^i$ has order 2: if an integer $n$ is not divisible by $3^i$, then neither is $n+n$.

\begin{notation}\label{nota:A.1} Let $G_i$ be groups isomorphic to $\mathbb{Z}/3^i$ $(i\in \mathbb{N})$ such that for $i\not=j$ the group $G_i\cap G_j$ is trivial (the unit).
\end{notation}

 We use these groups in the following result:

\begin{theorem}[\cite{Olshanskii1991}, Theorem 35.1]
Let $G_i$ ($i \in \mathbb{N}$) be nontrivial finite groups without elements of order 2, having pairwise trivial intersections. There exists a group $G$ with the following properties:
\begin{enumerate}[(1)]
    \item $G$ contains   $G_i$ as a subgroup for every $i \in \mathbb{N}$.
    \item Every proper subgroup of $G$ is finite.
    \item $G$ has two generators.
\end{enumerate}
\end{theorem}

Item (2) is formulated in loc. cit in more detail (that we do no employ below): for a sufficiently large number $n\in \mathbb{N}$, every proper subgroup of $G$ is either cyclic of order dividing $n$, or its conjugate to $G_i$ (thus, isomorphis to $G_i$) for some $i\in \mathbb{N}$.

\begin{corollary}
Let $G$ be the group constructed above for $G_i$ in Notation \ref{nota:A.1}. Then $G$ has arbitrarily long finite chains of subgroups. Moreover, it has no infinite chains of subgroups.
\end{corollary}

Indeed, since $A_i$ contains a chain of subgroups of length $i+1$, so does $G$. Every descending chain of subgroups
\(
H_0 \supset H_1 \supset H_2 \supset \dots
\)
is finite because $H_1 (\subset G)$ is finite. Every ascending chain
\(
H_0 \subset H_1 \subset H_2 \subset \dots
\)
is finite because otherwise its union (which is itself a subgroup) is all of $G$. Thus, the two generators of $G$ lie in $H_i$ for some $i$. This means that $H_i = G$.

\bibliographystyle{abbrv}
\bibliography{references}

\medskip

\begin{tabular}{p{14cm}}
Department of Mathematics, Faculty of Electrical Engineering, \\
Czech Technical University in Prague, Czech Republic. \\

Institute for Theoretical Computer Science, Technical University of Braunschweig, Germany.\\
Email: \href{mailto:j.adamek@tu-bs.de}{\texttt{j.adamek@tu-bs.de}}
\end{tabular}

\bigskip

\begin{tabular}{p{14cm}}
University of Coimbra, CMUC, Department of Mathematics, Portugal.\\
Polytechnic University of Viseu, ESTGV, Portugal

Email: \href{mailto:sousa@estv.ipv.pt}{\texttt{sousa@estv.ipv.pt}}
\end{tabular}

\end{document}